\documentclass{article}
\usepackage[utf8]{inputenc}
\usepackage[english]{babel}
\usepackage{amsmath,amsthm,amssymb,amsfonts}
\usepackage{tikz}
\usepackage[section]{placeins}
\usepackage{graphicx}
\usepackage{hyperref}
\graphicspath{ {images/} }
\usepackage{multicol}

\theoremstyle{plain}
\newtheorem{thm}{Theorem}[section]
\newtheorem{prop}[thm]{Proposition}
\newtheorem{lem}[thm]{Lemma}

\theoremstyle{definition}

\title{Reliability Polynomials of Simple Graphs having Arbitrarily many Inflection Points}
\author{Danielle Blackwell, Christopher Hunt, Keyne\'e Johnson }
\date{\today}

\begin{document}

\maketitle

\begin{abstract}
In this paper we show that for each $n$, there exists a simple graph whose reliability polynomial has at least $n$ inflection points.
\end{abstract}

\section{Introduction}

The reliability of a graph $G$ is the probability that the graph remains connected when each edge is included, or ``functions", with independent probability $p$. Equivalently, we can say that each edge fails with probability $q = 1-p$.  This function can be written as a polynomial in either $p$ or $q$, though for our purposes it will be convenient to use $q$; for instance, if $f(q) = R(G)(q)$ for a nontrivial graph $G$, then we have $f(0) = 1$ and $f(1) = 0$. Since the first derivative of $R(G)(q)$ is always negative on $(0,1)$, it is natural to consider whether the second derivative is ever zero, i.e., whether $R(G)(q)$ has any inflection points.

It is typical for the reliability of a graph to have at least one inflection point, and families of simple graphs with reliability polynomials having two inflection points have been found \cite{BKK}.  In \cite{GM}, Graves and Milan show that there exist non-simple graphs whose reliability polynomials have at least $n$ inflection points for any integer $n$. They point out that no example is known of a simple graph whose reliability polynomial has more than two inflection points. What we show in this paper is that for each $n$, there exists a simple graph whose reliability polynomial has at least $n$ inflection points.
\section{Preliminaries}

Our proof consists of two major parts: we first demonstrate that there exist reliability polynomials whose second derivative satisfies certain bounds, and then from this collection of polynomials we form products which have arbitrarily many inflection points. The \emph{one-point union} of graphs $G$ and $H$, denoted $G*H$, is the graph union where exactly one vertex is chosen from each graph and the chosen vertices are identified. Regardless of the choices made the reliability polynomial of the one-point union of graphs is the product of the reliability polynomials of each graph. Ultimately, the graphs we use will be one-point unions of complete graphs, and in the following section we demonstrate that the second derivative of such graphs can be made arbitrarily small outside a given interval. 

In this section we establish a few facts about reliability polynomials in general.  The reliability polynomial of a graph $G$ can be written in the form \begin{equation}\label{reldefn}
R(G)(q) = \sum_{i=0}^m N_i (1-q)^i q^{m-i},
\end{equation} where $m$ is the number of edges in $G$, and $N_i$ counts the number of connected spanning subgraphs of $G$ with $i$ edges.

We first prove a fact about all polynomials having a form similar to \eqref{reldefn}.
\begin{prop}\label{derivbound} Let $f(q) = \sum_{i=0}^m N_i(1-q)^iq^{m-i}$, where the $N_i$ are either all non-negative or all non-positive.  Then for $q \in [0,1]$,
\[q(1-q)|f'(q)| \leq m|f(q)|.\]
\end{prop}
\begin{proof}
We first compute

\begin{equation}\label{relderiv}
f'(q) = \sum_{i=0}^m N_i ((m-i)(1-q)^i q^{m-i-1} - i(1-q)^{i-1}q^{m-i}).
\end{equation}

Then we have
\begin{gather*}
q(1-q)|f'(q)| = \left|\sum_{i=0}^m (m-i)N_i(1-q)^{i+1} q^{m-i} - iN_i(1-q)^{i}q^{m-i+1}\right|\\
\leq m(1-q)\left|\sum_{i=0}^m N_i(1-q)^iq^{m-i}\right| + mq\left|\sum_{i=0}^m N_i(1-q)^iq^{m-i}\right| \leq m|f(q)|.
\end{gather*}

Note that the first inequality, where we bound the coefficients uniformly by $m$, uses the hypothesis that the $N_i$ have the same sign.
\end{proof}

In the form \eqref{reldefn} of a reliability polynomial, the coefficient $N_i$ counts a subset of the subgraphs of $G$ with $i$ edges, and thus we have $0 \leq N_i \leq \binom{m}{i}$.  Then clearly the above theorem applies when $f$ is a reliability polynomial.  However, we can also consider

\begin{equation}
(1-R(G))(q) = \sum_{i=0}^m \left(\binom{m}{i} - N_i\right)(1-q)^i q^{m-i}.
\end{equation}

By the above observation, the coefficients of this polynomial are non-negative, and so \hyperref[derivbound]{Lemma \ref*{derivbound}} applies to $1-R(G)$ as well.  Next we show that it can also be applied to $R(G)'(q)$.

\begin{prop}
Let $f(q) = \sum_{i=0}^m N_i(1-q)^iq^{m-i}$ be a reliability polynomial. Then $f'(q)$ can be written in the same form, and all of the coefficients are non-positive. 
\end{prop}
\begin{proof}
We already computed the derivative of such a function in \eqref{relderiv}; collecting like terms gives

\begin{equation}\label{relderivcollect}
f'(q) = \sum_{i=0}^{m-1} ((m-i)N_i - (i+1)N_{i+1})(1-q)^iq^{m-i-1}.
\end{equation}

Recall that $N_i$ represents the number of connected spanning subgraphs of $G$ with $i$ edges.  Thus we can think of $(m-i)N_i$ as counting the pairs consisting of a connected spanning subgraph of size $i$ together with a particular edge not in the subgraph.  Similarly, $(i+1)N_{i+1}$ counts the number of pairs consisting of a connected spanning subgraph of size $i+1$ and an edge in the subgraph.  However, since adding an edge to a connected spanning graph gives another connected spanning graph, $(m-i)N_i \leq (i+1) N_{i+1}$. Thus each of the coefficients in the expression for $f'(q)$ is non-positive.
\end{proof}

We note that, by the preceding argument, the coefficient of $(1-q)^i q^{m-i-1}$ in $R(G)'(q)$ counts the number of pairs consisting of a connected spanning subgraph of size $i+1$ and a bridge in the subgraph.  This gives the following result.

\begin{prop}\label{endpointderiv}
Let $f(q)$ be the reliability polynomial of a graph $G$.  If $G$ has no bridges, then $f'(0) = 0$. If $G$ has at least 3 vertices, then $f'(1) = 0$.
\end{prop}
\begin{proof}
From \eqref{relderivcollect}, we see that $f'(0) = 0$ when the coefficient of $(1-q)^{m-1}$ is zero.  Since there is only a single subgraph of size $m$, namely the graph $G$, this is equivalent to $G$ having no bridges.

Similarly, $f'(1) = 0$ whenever the coefficient of $q^{m-1}$ is zero.  In a subgraph with one edge, that edge is a bridge; thus, $f'(1) = 0$ if and only if there are no connected spanning subgraphs with one edge, which is clearly true if the graph $G$ has 3 or more vertices.
\end{proof}

\section{Bounding the Reliability of Complete Graphs}

For the first half of the proof, we will work directly with the reliability polynomials of complete graphs.  We make use of a recurrence relation given in Colbourn's book \cite{Colbourn}, which we restate here.  If $r_n(q)$ is the reliability of the complete graph $K_n$, then

\begin{align*}
r_1 &= 1;\\
r_n &= 1 - \sum_{k=1}^{n-1} \binom{n-1}{k-1}q^{k(n-k)}r_k.
\end{align*}

To make these polynomials easier to work with, we bound them by simpler polynomials.

\begin{lem} Let $\alpha = \frac{1}{8}$, and let $r_n$ denote $R(K_n)$.  For $q \in [0,\alpha]$ and $n \geq 2$, we have
\[1 - (n+1)q^{n-1} \leq r_n(q) \leq 1 - (n-1)q^{n-1}.\]
\end{lem}
\begin{proof}
We proceed by induction on $n$.
The base case where $n=2$ is clear, since $r_2 = 1-q$.
Now, suppose the claim is true for $2 \leq k \leq n-1$.  We first note that, since $r_n = 1 - \sum_{k=1}^{n-1} \binom{n-1}{k-1}q^{k(n-k)}r_k$, we can rewrite the inequality we would like to prove as
\[(n-1)q^{n-1} \leq \sum_{k=1}^{n-1} \binom{n-1}{k-1}q^{k(n-k)}r_k \leq (n+1)q^{n-1}.\]
We will first prove the left hand side of the inequality.  From the induction hypothesis, $r_{n-1} \geq 1 - nq^{n-2}$, and the fact that $r_k \geq 0$ for all $k$, we have
\[\sum_{k=1}^{n-1} \binom{n-1}{k-1}q^{k(n-k)}r_k \geq q^{n-1} + (n-1)q^{n-1}(1 - nq^{n-2}).\]
If we can show that $q^{n-1}(1 - n(n-1)q^{n-2}) \geq 0$, it will follow that the right hand side is greater than $(n-1)q^{n-1}$, which was what we wanted.  If we suppose $q \leq \frac{1}{6}$, then it follows that $n(n-1)q^{n-2} \leq 1$ for $n \geq 3$, and the claim holds.

We now proceed to the right hand side of the inequality.  Since $r_k \leq 1$ for all $k$, and $n(n-k) \geq 2n-4$ when $2 \leq k \leq n-2$, we have \begin{align*}
		\sum_{k=1}^{n-1} \binom{n-1}{k-1}q^{k(n-k)}r_k &\leq \sum_{k=1}^{n-1} \binom{n-1}{k-1}q^{k(n-k)}\\
			&\leq nq^{n-1} + \sum_{k=2}^{n-2} \binom{n-1}{k-1}q^{k(n-k)} \\
			&\leq nq^{n-1} + 2^{n-1} q^{2n-4}.
\end{align*}
To show that this is less than or equal to $(n+1)q^{n-1}$, it suffices to show that $2^{n-1}q^{n-3} \leq 1$.  If $n \geq 4$, then it is enough to take $q \leq \frac{1}{8}$.  We may verify directly that
\[
r_3 = 1 - q^2 - 2q^2(1 - q) = 1 - 3q^2 + 2q^3 \leq 1 - 2q^2
\]
provided $q \leq \frac{1}{2}$, which is clearly satisfied since $q \leq \frac{1}{8}$.

This completes our induction proof, and so the bounds hold for all $n \geq 2$.
\end{proof}

We also prove another bound which will be used during the proof of the following theorem.

\begin{lem}\label{roverp2} Let $r_n(q) = R(K_n)(q)$, with $n \geq 3$.  Then for $q \in [0,1]$,
\[\frac{r_n(q)}{(1-q)^2} \leq \frac{1}{2}\binom{n}{2}^2.\]
\end{lem}
\begin{proof}
Let $m = \binom{n}{2}$ denote the number of edges of $K_n$.  For $2 \leq i \leq m$, we have
\[\binom{m}{i} = \frac{m(m-1)}{i(i-1)}\binom{m-2}{i-2} \leq \frac{m^2}{2}\binom{m-2}{i-2}.\]
Recall that we can write $r_n(q) = \sum_{i=0}^{m} N_i (1-q)^iq^{m-i}$.  Since $N_i$ is the number of connected spanning subgraphs of $K_n$ of size $i$, we have $0 \leq N_i \leq \binom{m}{i}$ for all $i$, and $N_i = 0$ for $0 \leq i < n-1$.  Since $n \geq 3$, we can say
\begin{gather*}
\frac{r_n(q)}{(1-q)^2} = \sum_{i=0}^{m-2} N_{i+2}(1-q)^iq^{m-i-2} \leq \sum_{i=0}^{m-2} \binom{m}{i+2}(1-q)^iq^{m-i-2}\\
\leq \frac{m^2}{2} \sum_{i=0}^{m-2} \binom{m-2}{i}(1-q)^iq^{m-2-i} = \frac{m^2}{2} = \frac{1}{2}\binom{n}{2}^2.\qedhere
\end{gather*}
\end{proof}

Now we have enough in place to prove our first theorem.

\begin{thm} Let $0 < a < b < \frac{1}{8}$, and let $\epsilon > 0$.  Then there exists a graph $G$ such that $|R(G)''(q)| \leq \epsilon$ when $q$ in $[0,1] \setminus [a,b]$.
\end{thm}
\begin{proof}
We first recall that $r_n$ denotes the reliability of the complete graph $K_n$ and the inequalities 
\[1-(n+1)q^{n-1} \leq r_n(q) \leq 1-(n-1)q^{n-1}\]
are valid for $q \in [0,1/8]$.  Since we are only interested in upper bounds, we reformulate the left hand side as $1 - r_n(q) \leq (n+1)q^{n-1}$.

We compute
\[(r_n^\ell)''(q) = \ell(\ell-1)r_n^{\ell-2}(q)r_n'^2(q) + \ell r_n^{\ell-1}(q)r_n''(q),\]

from which we now obtain various bounds. Note that the coefficients of $1-r_n(q)$  are all non-negative. From our lemmas, we have
\begin{align*}
    &q(1-q)|r_n'(q)| \leq \binom{n}{2}|r_n(q)|;\\
&q(1-q)|r_n'(q)| = q(1-q)|(1-r_n)'(q)| \leq \binom{n}{2}|1-r_n(q)|;\\
&q(1-q)|r_n''(q)| \leq \binom{n}{2}|r_n'(q)|.
\end{align*}
For $q \in [0,1/8]$, we have
\begin{align*}
q^2(1-q)^2|(r_n^{\ell})''(q)| &\leq \ell(\ell-1)\frac{n^2(n-1)^2}{4}|r_n^{\ell-2}(q)||1-r_n(q)|^2\\
&\quad + \ell \frac{n^2(n-1)^2}{4}|r_n^{\ell-1}(q)||1-r_n(q)|\\
&\leq \frac{\ell^2n^6}{4}(1-(n-1)q^{n-1})^{\ell-2}q^{2n-2}\\
&\quad + \frac{\ell n^5}{4}(1-(n-1)q^{n-1})^{\ell-1}q^{n-1},
\end{align*}
where we have, in addition to using the lemmas, simplified $(n-1)(n+1) \leq n^2$.  We note that $(1-q)^2 \geq \frac{1}{4}$ when $q \leq \frac{1}{8}$, and so we may rearrange this inequality to obtain
\begin{align*}
|(r_n^{\ell})''(q)| &\leq \ell^2n^6(1-(n-1)q^{n-1})^{\ell-2}q^{2n-4}\\
&\quad +\ell n^5(1-(n-1)q^{n-1})^{\ell-1}q^{n-3}.
\end{align*}

We denote the two terms in the above bound by $f(q)$ and $g(q)$, respectively.  Note that $f$ and $g$ are both dependent on our choice of $\ell$ and $n$; from here on it will be useful to assume that $n \geq 4$.

We'd like to demonstrate that $f(q)+g(q) < \epsilon$ on $[0,a]$ and $[b,1/8]$.  To do this, we first show that $f+g$ is increasing on the first interval and decreasing on the second; then it suffices to show that $f(a)+g(a) < \epsilon$ and $f(b)+g(b) < \epsilon$.

Taking the derivative of $f$ and factoring, we see that the sign of $f'(q)$ is the same as that of the expression
\[-(\ell-2)(n-1)^2q^{n-1} + (2n-4)(1-(n-1)q^{n-1}).\]
This expression is non-negative precisely when
\[((\ell-2)(n-1)^2 + (2n-4)(n-1))q^{n-1} \leq 2n-4.\]
After some estimation, we see that if $\ell n(n-1)q^{n-1} \leq 2(n-2)$, then $f'(q) \geq 0$.  Since $n \geq 4$, it suffices to show $\ell nq^{n-1} \leq 1$.

The sign of $g'(q)$ is the same as that of
\[-(\ell-1)(n-1)^2q^{n-1} + (n-3)(1-(n-1)q^{n-1});\]
similarly, if $\ell n(n-1)q^{n-1} \leq n-3$, then $g'(q) \geq 0$.  Since $n \geq 4$, it suffices to show that $\ell nq^{n-1} \leq \frac{1}{4}$; if this is true, then the condition above holds as well, and so $(f+g)'(q) \geq 0$.  Moreover, since $\ell nq^{n-1} \leq \ell na^{n-1}$ for $q \in [0,a]$, it is sufficient to show that $\ell na^{n-1} \leq \frac{1}{4}$.

By the same sort of approximation, we see that if $\ell q^{n-1} \geq 1$, then $(f+~g)'(q) \leq 0$.  Again, if this is true at $b$, then it clearly holds for all $q \in [b,1/8]$ as well.

For $q \in [1/8,1]$, we have $\frac{1}{q^2} \leq 64$. Using our usual bounds, along with \hyperref[roverp2]{Lemma \ref*{roverp2}}, we see
\begin{align*}
|(r_n^{\ell})''(q)| &\leq \frac{1}{q^2} \ell^2 \binom{n}{2}^2 r_n^{\ell-1}(q) \frac{r_n(q)}{(1-q)^2}\\
&\leq 16\ell^2 \binom{n}{2}^4 r_n^{\ell-1}(q) \leq \ell^2n^8 r_n^{\ell-1}(q)\\
&\leq \ell^2n^8 r_n^{\ell-1}(1/8) \leq \ell^2n^8 (1-(n-1)(1/8)^{n-1})^{\ell-1}.
\end{align*}

Note that we were able to replace $q$ with $1/8$ since $q \in [1/8,1]$ and $r_n$ is decreasing.

We restate our conditions here: we want to find $\ell \geq 1$ and $n \geq 4$ such that
\begin{itemize}
\item $\ell na^{n-1} \leq \frac{1}{4}$,
\item $\ell b^{n-1} \geq 1$,
\item $f(a)+g(a) \leq \epsilon$,
\item $f(b)+g(b) \leq \epsilon$, and 
\item $\ell^2n^8 (1-(n-1)(1/8)^{n-1})^{\ell-1} \leq \epsilon$.
\end{itemize}

If we can find $\ell,n$ satisfying these conditions, then the arguments above show that $|R(K_n^{\ell})''(q)| \leq \epsilon$ for $q \in [0,1] \setminus [a,b]$.

To show that all of the above inequalities can be satisfied, we define a sequence of $\ell_i$ and $n_i$ such that the quantities on the left become arbitrarily small (or large, in the case of the second one) for sufficiently large $i$.  We begin by choosing positive integers $N,k$ such that $a < N^{-1/k} < b$; this is possible because there is an integer between $b^{-k}$ and $a^{-k}$ for sufficiently large $k$.  Then, we let $\ell_i = N^i$ and $n_i = ik$.

Consider the first expression: we rewrite
\[\ell_i n_i a^{n_i-1} = \frac{ik}{a}(Na^k)^i,\]
which becomes small as $i$ goes to infinity, by an application of l'Hospital's rule and the fact that $Na^k < 1$.  Similarly, we can write
\[\ell_i b^{n_i-1} \geq (Nb^k)^i,\]
which tends to infinity since $Nb^k > 1$.

For the next two inequalities, we analyze $f$ and $g$ separately.  First, we have
\begin{align*}
\log(f(q)) &= 2\log(\ell_i) + 6\log(n_i)  \\
&\quad +  (\ell_i-2)\log(1-(n_i-1)q^{n_i-1}) + (2n_i-4)\log(q).
\end{align*}
Note that $\log(1 - (n_i-1)q^{n_i-1}) \leq -(n_i-1)q^{n_i-1}$; since we'd like to show that $\log(f(q)) \to -\infty$ as $i \to \infty$, we may make this replacement.  Thus we must show that the expression
\begin{align*}
2i\log(N) + 6\log(ik) - (N^i-2)(ik-1)q^{ik-1} + 2i\log(q^k) - 4\log(q)
\end{align*}
becomes arbitrarily small as $i$ tends to infinity.  The term $4\log(q)$ is constant with respect to $i$, and so after some rearrangement, we need only consider
\begin{align*}
&2i\log(N) + 6\log(ik) - (N^i-2)(ik-1)q^{ik-1} + 2i\log(q^k)\\
&\leq 2i\log(Nq^k) + 6\log(ik) - N^{-1}(Nq^k)^i
\end{align*}
For $q=a$, the first term tends to negative infinity, and dominates the second, while the third term is also negative; thus $\log(f(a))$ can be made arbitrarily small, and $f(a) < \frac{\epsilon}{2}$ for sufficiently large $i$.  For $q=b$, we have $Nq^k > 1$, and so the third term dominates the first two terms, and again we can make $f(b)$ arbitrarily small.

Proceeding similarly, we collect the nonconstant terms of $\log(g(q))$, and make the same upward approximation as before :
\begin{align*}
\log(\ell) &+ 5\log(n) - (\ell-1)(n-1)q^{n-1} + n\log(q)\\
&= i\log(Nq^k) + 5\log(ik) - (N^i-1)(ik-1)q^{ik-1}\\
&\leq i\log(Nq^k) + 5\log(ik) - N^{-1}(Nq^k)^i.
\end{align*}
The arguments showing that $g(a)$ and $g(b)$ become arbitrarily small for sufficiently large $i$ are identical to those given above for $f$.

Finally, we consider the expression in the fifth inequality.  We have
\begin{align*}
\log(\ell^2&n^8 (1-(n-1)(\tfrac{1}{8})^{n-1})^{\ell-1})\\
&\leq 2i\log(N) + 8\log(ik) -(N^i-1)(ik-1)(\tfrac{1}{8})^{ik-1}\\
&\leq 2i\log(N) + 8\log(ik) -N^{-1}(N(\tfrac{1}{8})^k)^i.
\end{align*}

Since $N(\frac{1}{8})^k > 1$, the last term dominates the others, and evidently this can also be made arbitrarily small.

Thus, if we take $i$ sufficiently large, we see that $G = K_{n_i}^{\ell_i}$ satisfies the desired condition.
\end{proof}

\section{Main Result}

Now we proceed to the proof that there are reliability polynomials with arbitrarily many inflection points.  We choose a collection of intervals 
\[I_{k,m} = (a_{k,m},b_{k,m}) \subset [0,1],\]
for $k \geq 0$ and $m \geq 1$, such that

\begin{itemize}
\item $0 < a_{0,1} < b_{0,1} < a_{1,1} < b_{1,1} < a_{2,1} < \cdots$, and $b_{k,1} < \frac{1}{8}$ for all $k \geq 0$;
\item $I_{k,m+1} \subset I_{k,m}$;
\item $\ell(I_{k,m}) = b_{k,m} - a_{k,m} \leq 2^{-m}$.
\end{itemize}

For $n \geq 3$, $m \geq 1$, $K_n^m$ has at least three vertices and no bridges.  By our previous theorem, along with \hyperref[endpointderiv]{Lemma \ref*{endpointderiv}}, we can find a collection of reliability polynomials $s_{k,m}: [0,1] \to [0,1]$, satisfying the following properties.

\begin{itemize}
\item $s_{k,m}(0) = 1$ and $s_{k,m}(1) = 0$;
\item $s_{k,m}'(0) = s_{k,m}'(1) = 0$, and $s_{k,m}'(q) \leq 0$ for all $q$;
\item $|s_{k,m}''(q)| \leq 2^{-m-1}$ for $q \notin I_{k,m}$.
\end{itemize}

We now collect the consequences as a series of lemmas.

\begin{lem} If $q \leq a_{k,m}$, then $|s_{k,m}'(q)| \leq 2^{-m-1}$ and $|1 - s_{k,m}(q)| \leq 2^{-m-1}$.  If $q \geq b_{k,m}$, then $|s_{k,m}'(q)| \leq 2^{-m-1}$ and $|s_{k,m}(q)| \leq 2^{-m-1}$.
\end{lem}
\begin{proof} Consider the first statement; we suppose otherwise and apply the mean value theorem.  That is, we suppose there is a $c \in [0,a_{k,m}]$ such that $|s_{k,m}'(c)| \geq 2^{-m-1}$.  But this implies that there is a $d \in [0,c]$ with
\[|s_{k,m}''(d)| = \frac{|s_{k,m}'(c) - s_{k,m}'(0)|}{|c - 0|} > 2^{-m-1},\]
contradicting our assumption about the $s_{k,m}$.

Similarly applying the mean value theorem again shows that if there was a $c \in [0,a_{k,m}]$ with $|s_{k,m}(c) - 1| \geq 2^{-m-1}$ there  would be a $d$ at which the above bound is not satisfied, another contradiction.

The arguments for $q \in [b_{k,m},1]$ are similar.
\end{proof}

Now we'd like to show how {\em large} the derivatives must be on the intervals $I_{k,m}$.  In particular, we'd like to find a single point in each interval at which both the first and second derivatives are ``sufficiently large".

\begin{lem} In each $I_{k,m}$, there is a point $q_{k,m}$ satisfying $s_{k,m}'(q_{k,m}) < -2^{m - 2}$ and $s_{k,m}''(q_{k,m}) > 2^{2m - 2}$.
\end{lem}
\begin{proof}
We note that for every $k \geq 0$, $m \geq 1$, we have $s_{k,m}(a_{k,m}) \geq \frac{3}{4}$ and $s_{k,m}(b_{k,m}) \leq \frac{1}{4}$; in particular, the difference is at least $\frac{1}{2}$.  Then by the mean value theorem, there is a $c \in I_{k,m}$ where

\[|s_{k,m}'(c)| = \frac{|s_{k,m}(b_{k,m})-s_{k,m}(a_{k,m})|}{|b_{k,m}-a_{k,m}|} \geq \frac{2^{-1}}{2^{-m}} = 2^{m - 1}.\]

Note that since $s_{k,m}'(q) \leq 0$ for all $q$, we have $s_{k,m}'(c) = -2^{m-1}$.  Now let $d$ be the smallest number such that $d > c$ and $s_{k,m}'(d) = -2^{m - 2}$; applying the mean value theorem again shows that there is a $q_{k,m} \in [c,d] \subset I_{k,m}$ satisfying

\[s_{k,m}''(q_{k,m}) = \frac{s_{k,m}'(d) - s_{k,m}'(c)}{d - c} \geq \frac{2^{m-2}}{2^{-m}} = 2^{2m - 2}.\]

Note that by the definition of $d$, we must have $s_{k,m}'(q_{k,m}) \leq s_{k,m}'(d) = -2^{m - 2}$.  
\end{proof}

Now we show that there exist products of $s_{k,m}$ (that is, reliability functions of one point unions of complete graphs) with arbitrarily many inflection points.  This proof is modeled after the proof of the main result given in \cite{GM}.

\begin{thm} There exist functions $g_k : [0,1] \to [0,1]$ for $k \geq 0$, such that $g_k = s_{k,m_k} g_{k-1}$ for $k \geq 1$, and $g_k$ has at least $2k$ inflection points.
\end{thm}
\begin{proof}
We prove this by induction on $k$. We are going to show that there is a sequence of reliability polynomials $g_k$, integers $m_k$, and points $q_k \in I_{k,m_k}$, for $k \geq 0$, such that

\begin{itemize}
\item $g_k''(q_i) > 0$ for $i=0,...,k$;
\item $\int_{q_{i-1}}^{q_i} g_k''(q) dq < 0$ for $i=1,...,k$;
\item $g_k = s_{k,m_k} g_{k-1}$ if $k \geq 1$.
\end{itemize}

We begin with the base case $k=0$.  We simply let $g_0(q) = s_{0,1}$, $q_0 = q_{0,1}$.  By our lemma, $s_{0,1}''(q_{0,1}) > 0$, and so the first condition holds; the other conditions hold vacuously.

Now suppose that we have found a $g_{k-1}$ satisfying the above properties, and we would like to find an $m_k$ such that $g_k = s_{k,m_k} g_{k-1}$ also satisfies them.

We first consider

\begin{equation}
(s_{k,m}g_{k-1})' = s_{k,m}'g_{k-1} + s_{k,m}g_{k-1}'.
\end{equation}

For each $i < k$, $q_i < a_{k,1}$, so $\underset{m \to \infty}{\lim} s_{k,m}'(q_i) = 0$ and $\underset{m \to \infty}{\lim} s_{k,m}(q_i) = 1$.   It follows that
\[\underset{m \to \infty}{\lim} (s_{k,m}g_{k-1})'(q_i) = g_{k-1}'(q_i).\]

Since

\[\int_{q_{i-1}}^{q_i} (s_{k,m}g_{k-1})''(q) dq = (s_{k,m}g_{k-1})'(q_i) - (s_{k,m}g_{k-1})'(q_{i-1}),\]

we have

\[\lim_{m \to \infty} \int_{q_{i-1}}^{q_i} (s_{k,m}g_{k-1})''(q) dq = \int_{q_{i-1}}^{q_i} g_{k-1}''(q) dq < 0.\]

Now we look at $(s_{k,m}g_{k-1})'(q_{k,m})$, noting that the precise location of $q_{k,m}$ depends on $m$.  By construction, we have $s_{k,m}'(q_{k,m}) \leq -2^{m-2}$, but $|s_{k,m}(q_{k,m})| \leq 1$; thus, by taking $m$ sufficiently large, we can guarantee that $(s_{k,m}g_{k-1})'(q_{k,m}) - (s_{k,m}g_{k-1})'(q_{k-1}) < 0$.

Next we need to consider

\begin{equation}
(s_{k,m}g_{k-1})'' = s_{k,m}''g_{k-1} + 2s_{k,m}'g_{k-1}' + s_{k,m}g_{k-1}''.
\end{equation}

For $q < a_{k,1}$, we have $\underset{m \to \infty}{\lim} (s_{k,m}g_{k-1})''(q) = g_{k-1}''(q)$, since $\underset{m \to \infty}{\lim} s_{k,m}''(q) = 0$, $\underset{m \to \infty}{\lim} s_{k,m}'(q) = 0$, and $\underset{m \to \infty}{\lim} s_{k,m}(q) = 1$.  Thus by the induction hypothesis, for sufficiently large $m$ we have $(s_{k,m}g_{k-1})''(q_i) > 0$ for $i=0,...,k-1$.

Now we need only consider $(s_{k,m}g_{k-1})''(q_{k,m})$.  Since $g_{k-1}$ is a reliability polynomial, $g_{k-1}' \leq 0$, so the second term in the expansion of $(s_{k,m}g_{k-1})''$ is positive.  Since $s_{k,m} < 1$, the third term is bounded by the maximum of $|g_{k-1}''|$ on [0,1].  Finally, since $g_{k-1}$ is bounded away from 0 on $I_{k,1}$ and $\underset{m \to \infty}{\lim} s_{k,m}''(q_{k,m}) = \infty$, it follows that $(s_{k,m}g_{k-1})''(q_{k,m}) > 0$ for sufficiently large $m$.

If we let $m_k$ be sufficiently large such that all of the above constructions hold, and define $q_k = q_{k,m_k}$ and $g_k = s_{k,m_k} g_{k-1}$, then $g_k$ satisfies the induction hypothesis; thus we have constructed the desired $g_k$ for all $k \geq 0$.  We know that $g_k''(q_{i-1}) > 0$ and $g_k''(q_i) > 0$ for each $i = 1, \dots ,k$; but $\int_{q_{i-1}}^{q_i} g_k'' dq < 0$ tells us that $g_k''$ is negative somewhere on $[q_{i-1},q_i]$, which implies that there are at least 2 inflection points on this interval.  Thus $g_k$ has at least $2k$ inflection points, which was what we wanted to show.
\end{proof}

Note that the polynomial $g_k$ constructed in the above theorem is the product of the reliability polynomials of a number of complete graphs, and thus it is the reliability polynomial of the one-point union of those graphs. In Figure 1, the second derivatives of three such graphs are shown, demonstrating reliability polynomials of simple graphs having 3,4, and 5 inflection points.

\begin{center}
\begin{figure}
\includegraphics[width=6cm]{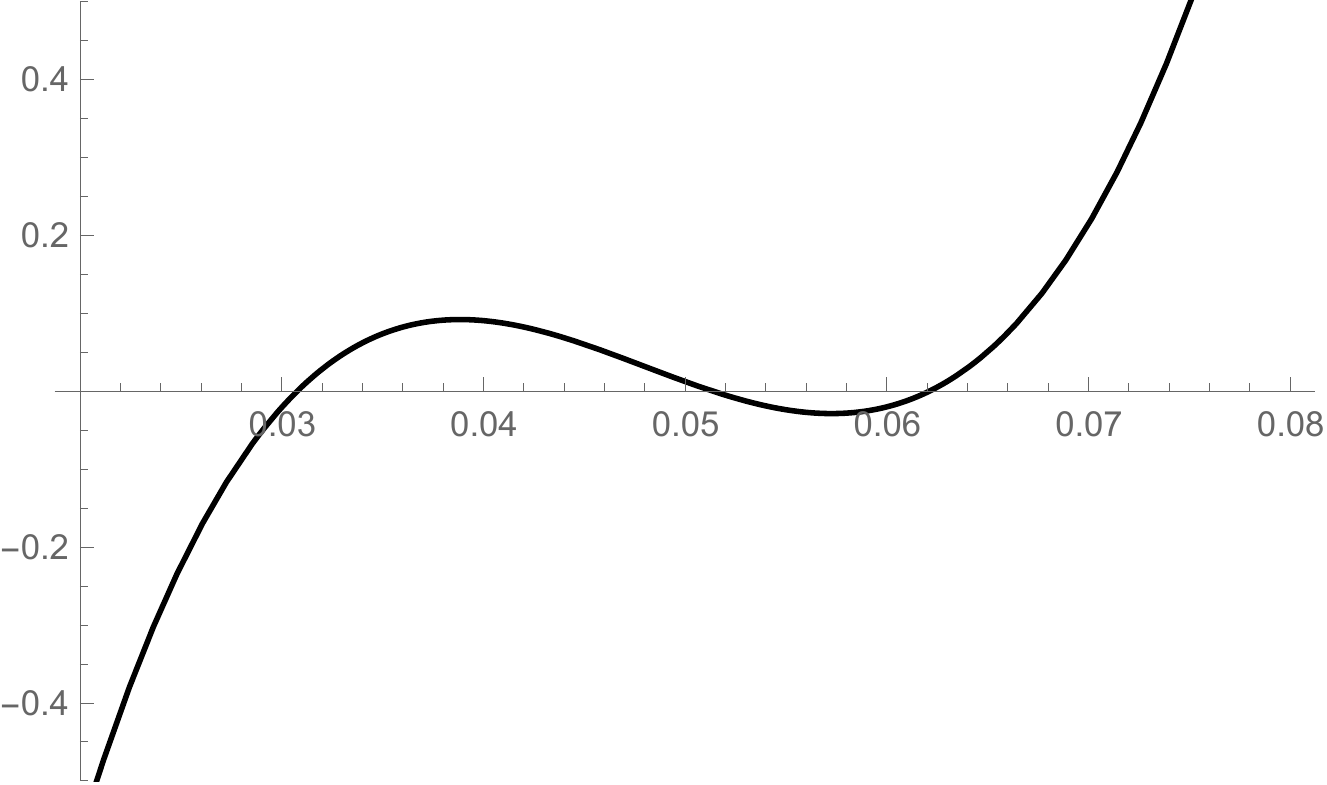}
\includegraphics[width=6cm]{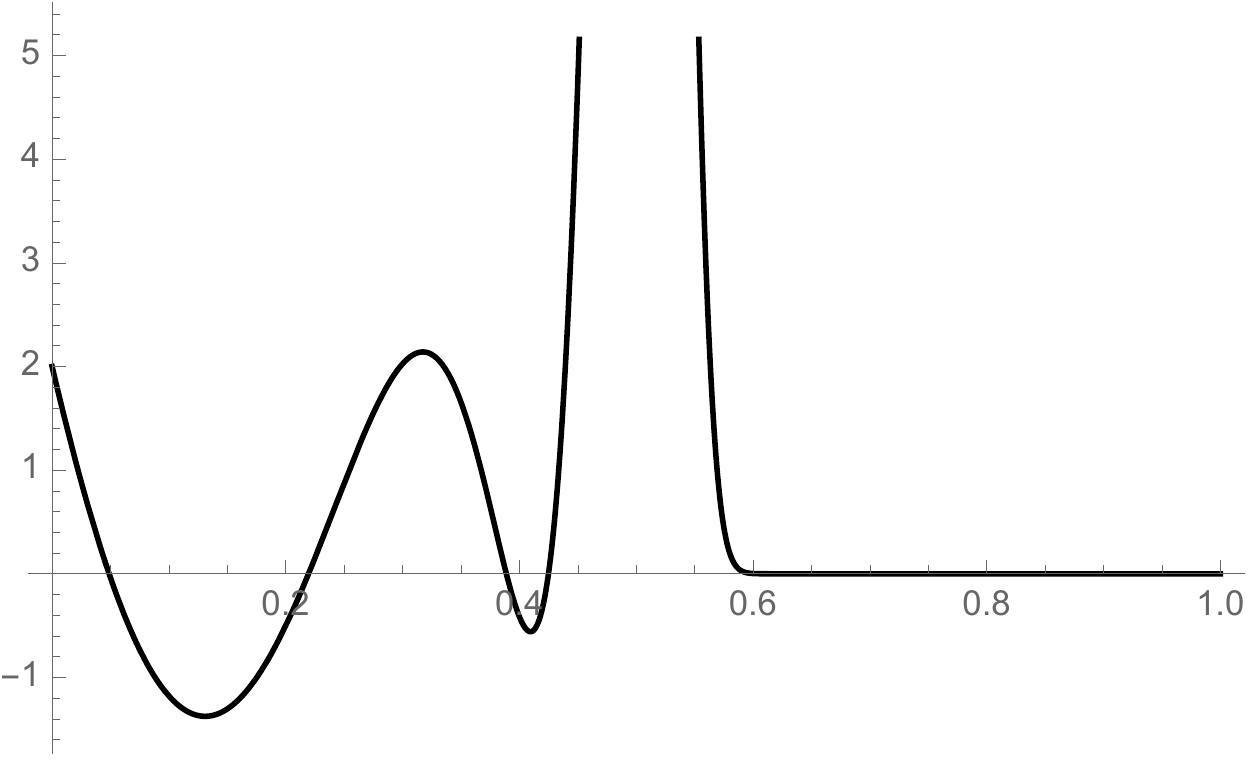}
\includegraphics[width=6cm]{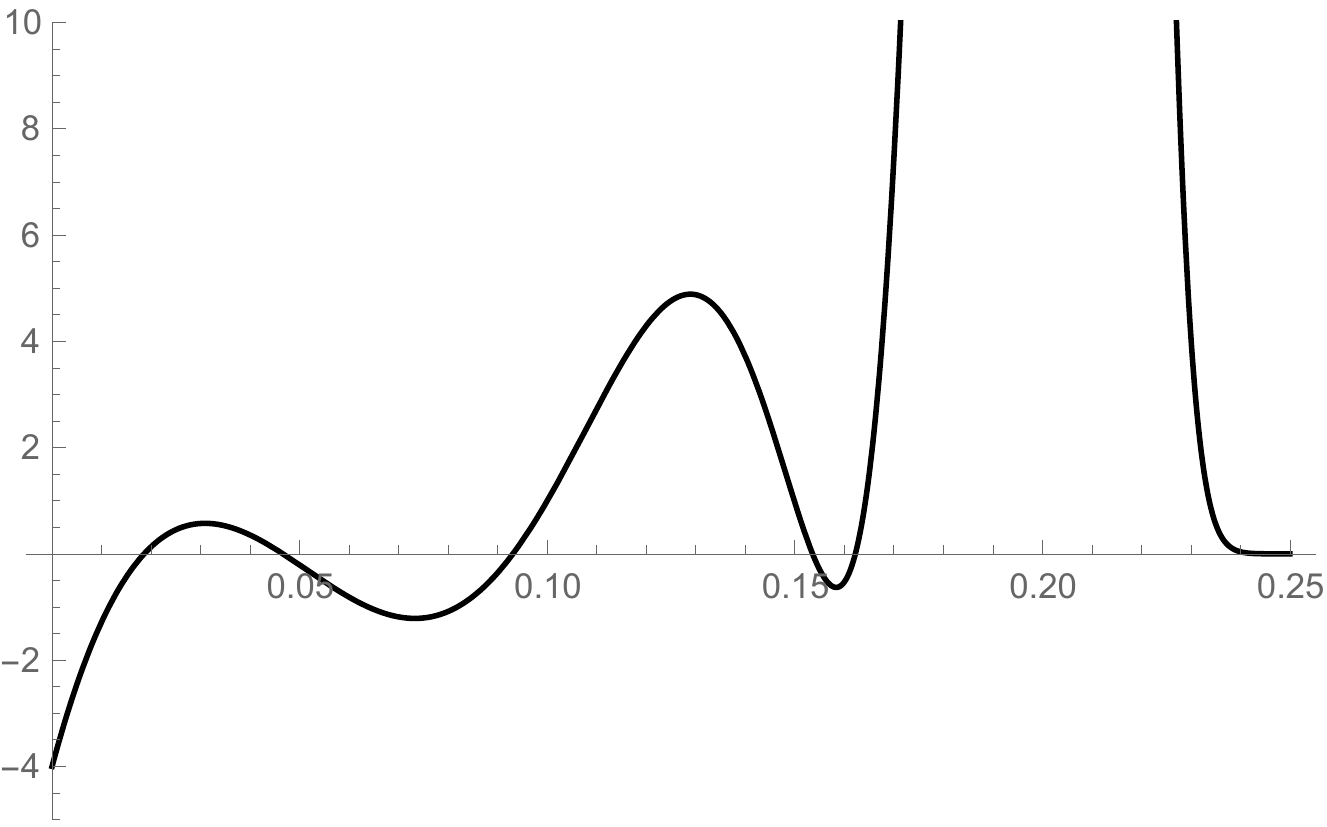}
\caption{\label{Examples} The graphs of the second derivatives of $R(K_{2}^{5} * K_{3}^{4} * K_{4}^{3} * K_{5}^{92}), R(K_{2}^{2} * K_{4}^{2} * K_{14}^{750}),$ and $R(K_{2}^{5} * K_{3}^{4} * K_{5}^{116} * K_{14}^{100,000,000})$ respectively.}
\end{figure}
\end{center}

\section{Acknowledgements}
This research was conducted during the 2015 Research Experience for Undergraduates at University of Texas at Tyler under the direction of Dr. David Milan. We would like to thank the organizers of the UT Tyler REU for their guidance during this project. This REU was supported by the National Science Foundation (DMS-1359101).

\bibliographystyle{amsplain}

\end{document}